\documentclass[10pt]{iosart2c}

\usepackage[none]{hyphenat}

\usepackage[T1]{fontenc}

\usepackage[utf8]{inputenc}
\usepackage{times}%
\usepackage{float,lscape}
\usepackage{multicol}

\usepackage{times}
\usepackage{amsmath,amsthm}
\usepackage{dcolumn}

\newtheorem{ex}{Example}[subsection]
\newtheorem{nt}{Note}[subsection]

\newtheorem{thm}{Theorem}[section]
\newtheorem{defn}{Definition} [subsection]
\newcolumntype{d}[1]{D{.}{.}{#1}}

\firstpage{1} \lastpage{13} \volume{1} \pubyear{2015}

\begin{document}

\begin{frontmatter}                         

\hyphenpenalty
\exhyphenpenalty

\title{Total Ordering Defined on the set of all Intuitionistic Fuzzy Numbers } 

\runningtitle{Total Ordering Defined on the set of all Intuitionistic Fuzzy Numbers}

\author[A]{\fnms{V. Lakshmana Gomathi Nayagam} \snm{} }, 
\author[A]{\fnms{Jeevaraj. S} \snm{}\thanks{Corresponding author. Tel.:+91-9788868172;\\ E-mail: jeevanitt@gmail.com}}
and
\author[B]{\fnms{Geetha Sivaraman} \snm{}}
\runningauthor{Lakshmana Gomathi Nayagam et al.}
\address[A]{Department of Mathematics, National Institute of Technology Tiruchirappalli, India,  \\ E-mail: velulakshmanan@nitt.edu\\ }
\address[B]{Department of Mathematics, St.Joseph's College (Autonomous), Tiruchirappalli, India, \\Email:geedhasivaraman@yahoo.com}

\begin{abstract}
 L.A.Zadeh  introduced the concept of fuzzy set theory as the generalisation of classical set theory in 1965 and further it has been generalised to intuitionistic fuzzy sets (IFSs) by Atanassov in 1983  to model information by the membership, non membership and hesitancy degree more accurately than the theory of fuzzy logic.  The notions of intuitionistic fuzzy numbers in different contexts were studied in literature and applied in real life applications.  Problems in different fields involving qualitative, quantitative and uncertain information can be modelled better using  intutionistic fuzzy numbers introduced in \cite{13} which generalises intuitionistic fuzzy values \cite{1,7,13}, interval valued intuitionistic fuzzy number (IVIFN) \cite{9} than with usual IFNs \cite{Df,18,Ne}.  Ranking of fuzzy numbers have started in early seventies in the last century and a complete ranking on the class of fuzzy numbers have achieved by W.Wang and Z.Wang only on 2014. A complete ranking on the class of IVIFNs, using axiomatic set of membership, non membership, vague and precise score functions has been introduced and studied by Geetha et al.\cite{9}. In this paper, a total ordering on the class of IFNs \cite{13} using double upper dense sequence in the interval $[0,1]$ which generalises the total ordering on fuzzy numbers (FNs) is proposed and illustrated with examples. Examples are given to show the proposed method on this type of IFN is better than existing methods and this paper will give the better understanding over this new type of IFNs.
\end{abstract}

\begin{keyword}
Double upper dense sequence \sep total order relation \sep intuitionistic fuzzy number \sep interval Valued intuitionistic fuzzy number \sep trapezoidal intuitionistic fuzzy number.
\end{keyword}

\end{frontmatter}

\section{Introduction}
$~~~~~$Information system (IS) is a decision model used to select the best alternative from all the alternatives in hand under various attributes. The data collected from the experts may be incomplete or imprecise numerical quantities. To deal with such data, the theory of IFS provided by Atanassov \cite{1} aids better. In information system, dominance relation rely on ranking of data, ranking of intuitionistic fuzzy numbers is inevitable.

$~~~~~$ Many researchers have been working in the area of ranking of IFNs since last century. Different ranking methods for intuitionistic fuzzy values, interval valued intuitionistic fuzzy numbers have been studied in \cite{Ch,8,Dh,123,11,Lin,Ll,Li,Jw,25,Ju,Z,Hz}.  But till date, there exists no single method or combination of methods available to rank any two arbitrary IFNs.  The difficulty of defining total ordering on the class of  intuitionistic fuzzy numbers is that there is no effective tool to identify an arbitrarily given intuitionistic fuzzy number by finitely many real-valued parameters. In this work, by establising a new decomposition theorem for IFSs, any IFN can be identified by infinitely many but countable number of parameters. A new decomposition theorem for intuitionistic fuzzy sets is established by the use of an double upper dense sequence defined in $[0,1]$. Actually there are many double upper dense sequences available in the interval $[0,1]$. Since the choice of a double upper dense sequence is considered as the necessary reference systems for defining a complete ranking, infinitely many total orderings on the set of all IFNs can be well defined based on each choice of double upper dense sequence.  After introduction, some necessary fundamental knowledge on ordering and intuitionistic fuzzy numbers, interval valued intuitionistic fuzzy numbers is introduced in Section 2. In Section 3, a new decomposition theorem for intuitionistic fuzzy sets is establised using double upper dense sequence defined in the interval $[0,1]$. Section 4 is used to define total ordering on the set of all intuitionistic fuzzy numbers by using double upper dense sequence in the interval $[0,1]$. Several examples are given in Section 5 to show how the total ordering on IFNs can be used for ranking better than some other existing methods. Application of our proposed method in solving intuitionistic fuzzy information system problem is shown in section 5 by developing a new algorithm. Finally conclusions are given in section 6.
\subsection{Motivation}
The capacity to handle dubious and uncertain data is  more effectively done by stretching out intuitionistic fuzzy values to TrIFNs because the membership and non-membership degrees are better expressed  as\\ trapezoidal values rather than exact values. TrIFNs are generalisation of intuitionistic fuzzy values and IVIFNs.  As a generalisation, the set of TrIFNs should contain the set of all intuitionistic fuzzy values and IVIFNs. But the existing defintion for TrIFNs \cite{Df,18,Ne} does not contain the set of intuitionistic fuzzy values which means the existing definition for TrIFN is not the real generalisation of  intuitionistic fuzzy values. Hence the study about new structure for intuitionistic fuzzy number \cite{13} is essential. In the application point of view our proposed method is more applicable and more natural when it is compared  with the existing methodology.
More precisely,  
the existing definition for Trapezoidal intuitionistic fuzzy numbers (TrIFN) present in the literature \cite{Df,18,Ne} is defined as $A=\langle(a,b,c,d)(e,f,g,h)\rangle ~with~e \leq a \leq f \leq b \leq c \leq g \leq d \leq h$ does not generalise even the intuitionistic fuzzy value  of the kind $A=(a,c) ~with~ a+c \leq 1 ~and ~ a < c$. That is, if we write $A=(a,c)$ in trapezoidal intuitionistic form, we get $A=\langle(a,a,a,a)(c,c,c,c)\rangle ~with~a < c $ which contradicts the above definition for TrIFNs. So till today the real generalization of intuitionistic fuzzy values and  interval valued intuitionistic fuzzy numbers have not been studied in detail. This problem motivate us to study this type of intuitionistic fuzzy numbers \cite{13} and its ordering principles for ranking. 
\section{ Preliminaries}

Here we give a brief review of some preliminaries.

\begin{defn} (Atanassov \cite{1}). Let $X$ be a nonempty set. An intuitionistic fuzzy set (IFS) $A$ in $X$ is defined by $A= (\mu_{A}, \nu_{A})$, where $\mu_{A}:X \to [0,1]$ and   $\nu_{A}:X \to [0,1]$ with the conditions  $0 \leq \mu_{A}(x) + \nu_{A}(x) \leq 1, \forall x \in X$. The numbers $\mu_{A}(x), \nu_{A}(x) \in [0, 1]$ denote the degree of membership and non-membership of $x$ to lie in $A$ respectively. For each intuitionistic fuzzy subset $A$ in $X$,  $\pi_{A}(x) = 1 - \mu_{A}(x) - \nu_{A}(x)$ is called hesitancy degree of $x$ to lie in $A$. 
\end{defn}
\begin{defn} (Atanassov $\&$ Gargov, \cite{3}). Let $D[0,1]$ be the set of all closed subintervals of the interval $[0,1]$ . An interval valued intuitionistic fuzzy set on a set $X\neq \phi $ is an expression given by  $ A= \left\{ \left\langle x, \mu_A(x),\nu_A(x)\right\rangle:x \in X\right\}$ where \\ $\mu _A:X\rightarrow D[0,1],\nu _A:X\rightarrow D[0,1]$ with the condition $0<sup_x \mu_A (x)+sup_x \nu_A (x)\leq 1$.
\label{def1}
\end{defn}

The intervals $\mu_A (x)$ and $\nu_A (x)$ denote, respectively, the degree of belongingness and non-belongingness of the element $x$ to the set $A$. Thus for each $x\in X$, $\mu_A (x)$ and $\nu_A (x)$ are closed intervals whose lower and upper end points are, respectively, denoted by $\mu_{A_L} (x)$, $\mu_{A_U} (x)$ and $\nu_{A_L} (x)$, $\nu_{A_U} (x)$. We denote $ A=  \{ \left\langle x, [\mu_{A_L}(x),\mu_{A_U}(x)], [\nu_{A_L}(x),\nu_{A_U}(x)]\right\rangle:\\x \in X \}$ where $0<\mu_A(x) + \nu_A(x)\leq1$.

For each element $x \in X$ , we can compute the unknown degree (hesitance degree) of belongingness $\pi_A (x)$ to $A$ as $\pi _A(x)=1-\mu _A(x)-\nu _A(x)=[1-\mu _{A_U}(x)-\nu _{A_U}(x),1-\mu _{A_L}(x)-\nu _{A_L}(x)]$.
We denote the set of all IVIFSs in $X$ by IVIFS($X$). An IVIF value is denoted by $A=([a,b],[c,d])$ for convenience.
\begin{defn}
(Atanassov $\&$ Gargov, \cite{3}).
The complement $A^c$ of $A  = \left\langle  x, \mu_A (x), \nu _A (x) :x \in X\right\rangle$ is given by $A^c = \left \langle  x, \nu_A (x), \mu _A (x) :x \in X \right \rangle$.
\end{defn}
\begin{defn} (Lakshmana et.al \cite{13}).
 An intuitionis-\\tic fuzzy set (IFS) $A = (\mu_A, \nu_A)$ of $R$ is said to be an intuitionistic fuzzy number if
$\mu_A$ and $\nu_A$ are fuzzy numbers. Hence  $A = (\mu_A, \nu_A)$ denotes an intuitionistic fuzzy
number if $\mu_A$ and $\nu_A$ are fuzzy numbers with $\nu_A \leq {\mu_A}^{c}$,
where ${\mu_A}^{c}$ denotes the complement of $\mu_A$.
\label{def2}
\end{defn}

An intuitionistic fuzzy number \\$A=\{(a, b_1, b_2, c), (e, f_1, f_2, g) \}$ with $(e, f_1, f_2, g)\leq (a, b_1, b_2, c)^{c}$ is shown in fig(1).
\begin{figure}[ht]
\centering
\scalebox{0.6}{
\resizebox{8cm}{6cm}{\includegraphics{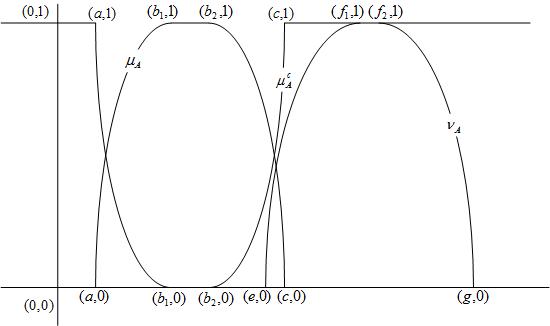}}}
\caption {Intuitionistic Fuzzy Number}
\end{figure}

\begin{defn}(Lakshmana et.al \cite{13})
A trapezoidal intuitionistic fuzzy number $A$ is defined by $A=\{ (\mu_{A}, \nu_{A}) \mid x \in R \}$, where $\mu_{A}$ and $\nu_{A}$ are trapezoidal fuzzy numbers with $ \nu_{A}(x) \leq \mu_{A}^{c}(x) $. \\We note that the condition $(e, f_1, f_2, g)\leq (a, b_1, b_2, c)^{c}$ of the trapezoidal intuitionistic fuzzy number $A = \{(a, b_1, b_2, c), (e, f_1, f_2, g) \}$ where 
$(a, b_1, b_2, c)$ and $(e, f_1, f_2, g)$ are membership and nonmembership fuzzy numbers of $A$ with  either $e \geq b_2$ and $f_1 \geq c$ or $f_2 \leq  a$ and $g \leq b_1$ on the legs of trapezoidal intuitionistic fuzzy number. 
\label{def3} 
\end{defn}

A trapezoidal intuitionistic fuzzy number $A=\{(a, b_1, b_2, c), (e, f_1, f_2, g) \}$ with $e \geq b_2$ and $f_1 \geq c$ is shown in fig(2).

\begin{figure}[ht]
\centering
\scalebox{0.6}{
 \resizebox{8cm}{6cm}{\includegraphics{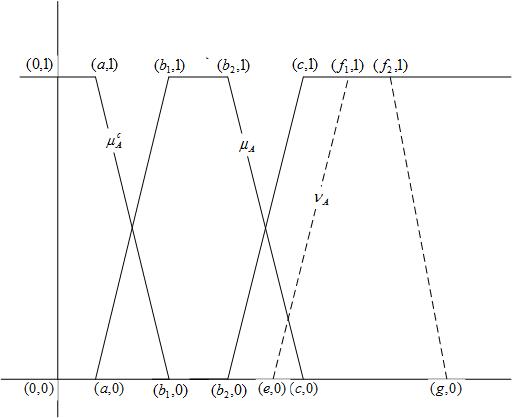}}
}

\caption {Trapezoidal Intuitionistic Fuzzy Number $A=\{(a, b_1, b_2, c), (e, f_1, f_2, g) \}$}
\label{f2}
\end{figure}
\begin{defn}
(Lakshmana et al., \cite{12})
For any IVIFN $A=([a,b],[c,d])$, the membership score function is defined as $L(A)=\frac{a+b-c-d+ac+bd }{2}  $.
\end{defn}
\begin{defn}
(Geetha et al., \cite{9})
For any IVIFN $A=([a,b],[c,d])$, the non-membership score function is defined as $LG(A)=\frac{-a-b+c+d+ac+bd }{2}  $.
\end{defn}
\begin{defn}
(Geetha et al., \cite{9})
For any IVIFN $A=([a,b],[c,d])$, the vague score function is defined as $P(A)=\frac{a-b-c+d+ac+bd }{2}  $. 
\end{defn}
\begin{defn}
(Geetha et al., \cite{9})
For any IVIFN $A=([a,b],[c,d])$, the imprecise score function is defined as $IP(A)=\frac{-a+b-c+d-ac+bd }{2}  $. 
\end{defn}
\begin{defn} Let $X$ be a non-empty set. Any subset of the cartesian product $X \times X$ is called a relation, denoted by $R$, on $X$. We write $aRb$ iff $(a,b) \in R$. A relation $R$ is called a partial ordering on $X$ if it is reflexive, antisymmetric, and transitive. A partial ordering $R$ on X is called total ordering if either $aRb$ or $bRa$ for any $a,b \in X$. Two total orderings $R_1 ~and~ R_2$ are different iff there exist $a,b \in X$ with $a \neq b$ such that $a R_1 b$ but $b R_2 a$ or $a  R_2 b$ but $b  R_1 a$. For any given total ordered infinite set, there are infintely many ways to redefine a new total ordering on it. A relation $ R$ is called an equivalence relation if it is reflexive, symmetric and transitive.
\end{defn}
\section{Total Ordering defined on Intuitionistic Fuzzy Number}

 $~~~~~$Geetha et.al \cite{9} have acheived the total ordering on the set of IVIFN using membership, nonmembership, vague and precise score functions. Let us recall the total ordering defined  by Geetha et.al \cite{9} on IVIFN. \\ Let $A=([a_1,b_1],[c_1,d_1]) , B = ([a_2,b_2],[c_2,d_2]) \in IVIFN$. The total ordering $\leq$ on IVIFN may be defined to one of the following criterion:\\
$(1).~ L(A)<L(B)$, or\\
$(2). ~ L(A)=L(B) ~but~ -LG(A)<-LG(B)$, or\\
$(3). ~ L(A)=L(B) ~and~ LG(A)=LG(B) ~but~ P(A)<P(B) $, or\\
$(4). ~ L(A)=L(B) , LG(A)=LG(B) ~and~ P(A)=P(B) ~but~ -IP(A) \leq -IP(B)$. \\
The above way of defining a total ordering is often referred to as lexicographic in literature \cite{15}. 
\\ In this section a new total ordering is defined in the set of all intuitionistic fuzzy numbers using the above ordering and a new decomposition theorem on intuitionistic fuzzy sets which is given in the following subsection.
\subsection{Upper Dense sequence in $(0,1]$ }
 $~~~~~$ The major difficulty in defining total ordering on the set of all fuzzy numbers is that there is no effective tool to identify an arbitrarily given fuzzy number by only finitely many real-valued parameters. To over come this difficulty, W.Wang, Z.Wang \cite{34} has introduced the concept of upper dense sequence in the interval $(0,1]$, and defined the total ordering on the set of all fuzzy numbers using this sequence. This upper dense sequence gives values for $\alpha$ in the $\alpha$-cut of FNs. But for IFNs, two sequences are needed to give values for $\alpha ~\&~ \beta$ in the $(\alpha, \beta)$-cut where $\alpha \in (0,1] ~\&~ \beta \in [0,1) $. For convenience upper dense sequence of $\alpha$ is denoted by $S_1=\left\{{{\alpha_i}| i=1,2,...}\right\}$ and upper dense sequence of $\beta$ is denoted by $S_2=\left\{{{\beta_i}| i=1,2,...}\right\}$. In this section, the upper dense sequence in $(0,1]$ is briefly reviewed.

\begin{defn}(W.Wang, Z.Wang, \cite{34})
A sequence $S = \left\{\alpha_i|i=1,2,...\right\}\subset (0,1]$ is said to be upper dense in $(0,1]$ if, for every point $x \in (0,1]$ and $\epsilon >0$, there exists some $ \alpha_i \in S$ such that $\alpha_i \in [x, x+\epsilon)$ for some $i$. A sequence $S \subset [0,1)$ is said to be lower dense in $[0,1)$ if, for every point $x \in [0,1)$ and any $\epsilon >0$, there exists $ \alpha_i \in S$ such that $\alpha_i \in (x-\epsilon , x]$ for some $i$. 
\label{defi}
\end{defn}
$~~~~~$From definition \ref{defi} we note that, any upper dense sequence in $(0,1]$ is nothing but a dense sequence with real number 1 and it is also a lower dense sequence.
\begin{defn} Let $S_1=\left\{ \alpha_i | i=1,2, ... \right\}$  and $S_2=\left\{ \beta_i | i=1,2, ... \right\}$ be two upper dense sequences in $(0,1]$ then the double upper dense sequence is  defined as $S=(S_1,S_2)=\left\{( \alpha_i, \beta_i ) | i=1,2, ... \right\}$.
\end{defn}
Example for an upper dense sequence and double upper dense sequence in $(0,1]$ is given as follows.
\begin{ex}
Let $S_1=\left\{ \alpha_i | i=1,2... \right\} $ be the set of all rational numbers in $(0,1]$, where $\alpha_1=1,\alpha_2=\frac{1}{2},\alpha_3=\frac{1}{3}, \alpha_4=\frac{2}{3}, \alpha_5=\frac{1}{4}, \alpha_6=\frac{3}{4}, \alpha_7=\frac{1}{5}, \alpha_8=\frac{2}{5}, \alpha_9=\frac{3}{5}, \alpha_{10}=\frac{4}{5},...$. \\ Then sequence $S_1$ is an upper dense sequence in $(0,1]$. If we allow a number to have multiple occurences in the sequence , the general members in upper dense sequence $S_ \beta=\left\{ s_{\beta_i}| i=1,2,...\right\}$ can be expressed by  $s_{\beta_i}=( \frac{i}{k}-\frac{k-1}{2}), i=1,2,...$ where $k=\left\lceil \sqrt{2i+\frac{1}{4}}-\frac{1}{2}\right\rceil$. That is, $s_{\beta_1}=1,~s_{\beta_2}=\frac{1}{2},~s_{\beta_3}=1,~s_{\beta_4}=\frac{1}{3},~s_{\beta_5}=\frac{2}{3},~s_{\beta_6}=1,~s_{\beta_7}=\frac{1}{4},~s_{\beta_8}=\frac{2}{4},~s_{\beta_2}=\frac{3}{4},...$. In sequence $S_{\beta}$, for instance, $s_{\beta_3}$ is the same real number as $s_{\beta_1}$. \label{exg}
\end{ex}
\begin{ex}
 Consider $S_1$ as in example \ref{exg} Let $S=(S_1,S_2)= \left\{ (\alpha_i, \beta_i) | \alpha_i= \beta_i,\alpha_i \in S_1 \right\}\\ =  \{ (1,1),(1/2,1/2),(1/3,1/3),(2/3,2/3), \\(1/4,1/4),(3/4,3/4),(1/5,1/5),... \} $. Clearly $S$ is a double upper dense sequence in $(0,1]$.
\label{eg}
\end{ex}
In the forthcoming sections this double upper dense sequence will be very useful for defining total orderings on the set of IFNs.

\subsection{Decomposition theorem for intuitionistic fuzzy number using upper dense sequence}
 $~~~~~$In this section, a new decomposition theorem for IFSs is established using the double upper dense sequence defined in $(0,1]$. Before establising a new decomposition theorem for intuitionistic fuzzy sets, it is needed for us to define decompostion theorems for intuitionistic fuzzy sets using special $\alpha-\beta$ cuts.
\begin{defn}
Let $X$ be a nonempty universal set and $A=\left(\mu _A,\nu _A\right)$ be an intuitionistic fuzzy set of $X$ with membership function $\mu _A$ and with a nonmembership function $\nu_A$. Let $\alpha, \beta \in (0, 1]$   
\begin{enumerate}
	\item The $\alpha - \beta$ cut, denoted by $^{\alpha - \beta}A$ is defined by $^{(\alpha - \beta)}A = ~^{\alpha}{\mu_A} \times ~^{\beta}{\nu_A}$ where $^{\alpha}{\mu_A} = \{x \in X|\mu_{A}(x)\geq \alpha\}$ and $^{\beta}{\nu_A} = \{x \in X|\nu_{A}(x)\geq \beta\}$. Equivalently $^{(\alpha - \beta)}A = \mu_{A}^{-1}([\alpha, 1]) \times \nu_{A}^{-1}([ \beta,1])$ is a subset of $ \wp(X) \times \wp(X)$.
	\item The  strong alpha beta cut, denoted by ${^{(\alpha - \beta)+}}A$ is defined by ${^{(\alpha - \beta)+}}A = ^{\alpha+}{\mu_A} \times ^{\beta+}{\nu_A}$ where $^{\alpha+}{\mu_A} = \{x \in X|\mu_{A}(x)>\alpha\}$ and $~^{\beta+}{\nu_A} = \{x \in X|\nu_{A}(x)>\beta\}$. Equivalently ${^{(\alpha - \beta)+}}A = \mu_{A}^{-1}((\alpha, 1]) \times \nu_{A}^{-1}(( \beta,1])$ is a subset of $\wp(X) \times \wp(X)$.
	\item The level set of $A$, denoted by $L(A)$ is defined by $L(A) = L_{\mu_A} \times L_{\nu_A}$ where \\$L_{\mu_A}=\left\{ \alpha | \mu_A(x) = \alpha ,~x \in X \right\} $ and \\$L_{\nu_A}=\left\{ \beta | \nu_A(x) = \beta ,~x \in X \right\} $ which is a subset of $\wp([0, 1]) \times \wp([0, 1])$. 
\end{enumerate}
\end{defn} 
 \begin{ex}
Let $A=\langle (0.17,0.3,0.47,0.56),\\(0.05,0.13,0.16,0.23) \rangle$ be a trapezezoidal intuitionistic fuzzy number.
Then $^{(\alpha - \beta)}A =([0.17+(0.3-0.17)\alpha, 0.56-(0.56-0.47)\alpha],[0.05+(0.13-0.05)\beta, 0.23-(0.23-0.16)\beta]) =([0.17+0.13\alpha, 0.56-0.09\alpha],[0.05+0.08\beta, 0.23-0.07\beta]),\forall \alpha,\beta \in [0,1]$.\\
${^{(\alpha - \beta)+}}A =((0.17+(0.3-0.17)\alpha, 0.56-(0.56-0.47)\alpha),(0.05+(0.13-0.05)\beta, 0.23-(0.23-0.16)\beta)) =((0.17+0.13\alpha, 0.56-0.09\alpha),(0.05+0.08\beta, 0.23-0.07\beta)),\forall \alpha,\beta \in [0,1]$.\\
$L_{\mu_A} =(\{{0.17+(0.3-0.17)\alpha, 0.56-(0.56-0.47)\alpha}\},\\ \{{0.05+(0.13-0.05)\beta, 0.23-(0.23-0.16)\beta)}\} \\=(\left\{{0.17+0.13\alpha, 0.56-0.09\alpha}\right\},\\ \{{0.05+0.08\beta, 0.23-0.07\beta}\}) ,\forall \alpha,\beta \in [0,1]$.
\end{ex}
One way of representing a fuzzy set is by special fuzzy sets on $\alpha$-cuts and another way of representing a fuzzy set is by special fuzzy sets on Strong $\alpha$-cuts. As a generalisaton of fuzzy sets, any intuitionistic fuzzy set can also be represented by the use of special intuitionistic fuzzy sets on $\alpha - \beta$ cuts and special intuitionistic fuzzy sets on strong $\alpha - \beta $ cuts.

\begin{defn}
The special intuitionistic fuzzy set $_{(\alpha,\beta)}{A}=\left(_{\alpha}{\mu_A},_{\beta}{\nu_A}\right)$ is defined by its membership ($_{\alpha}{\mu_A}$) and non-membership function ($_{\beta}{\nu_A}$) as follows,

$_{\alpha}{\mu_A}(x)=\begin{cases} \text\ \ \ \alpha ~ ~~~~~~~~~~,~ x \in ^{\alpha}{\mu_A}
\text\ \ \\~ \  \\~~\ 0 ~ ~~~~~~~~~~~ ,~otherwise \end{cases} 
$\\
$_{\beta}{\nu_A}(x)=\begin{cases} \text\ \ \ \beta ~ ~~~~~~~~~~,~ x \in ^{\beta}{\nu_A}
\text\ \ \\~ \  \\~~\ 0 ~ ~~~~~~~~~~~ ,~otherwise \end{cases} 
$. \\
\label{dfen1}
\end{defn}
The following decomposition theorems will show the representation of an arbitrary IFS in terms of the special IFSs $_{(\alpha,\beta)}{A}$. 
\begin{thm} 
{\bf First Decomposition theorem of an IFS}:
Let $X$ be a non-empty set. For an intuitionistic fuzzy subset $A=\left(\mu_A , \nu_A \right)$ in $X$, \\$A=\left({\bigcup _{\alpha  \in [0,1]}} ~~ _{\alpha } \mu_A , {\bigcup_{\beta  \in [0,1]}} _{\beta } \nu_A\right)$, where $\bigcup $ is standard union.
\label{CS1}
\end{thm}
\begin{proof}
Let  $x$ be an arbitrary element in $ X$ and  let $\mu_A (x)=a~\&~ \nu_A (x)=b$. \\ Then $\left( \left({\bigcup _{\alpha  \in [0,1]}} ~~ _{\alpha } \mu_A \right) (x) , \left({\bigcup_{\beta  \in [0,1]}} ~~ _{\beta } \nu_A\right)(x)\right)=\left( Sup_{\alpha  \in [0,1]} ~~ _{\alpha } \mu_A  (x) , Sup_{\beta  \in [0,1]} ~~ _{\beta } \nu_A(x)\right)$=$\\( max \left[Sup_{\alpha  \in [0,a]} ~~ _{\alpha } \mu_A (x) ,Sup_{\alpha  \in (a,1]} ~~ _{\alpha } \mu_A (x)\right],\\ max \left[ Sup_{\beta  \in [0,b]} ~~ _{\beta } \nu_A (x) , Sup_{\beta  \in (b,1]} ~~ _{\beta } \nu_A (x) \right] )$.
\\For each $\alpha \in [0,a]$, we have $\mu _A (x)=a \geq \alpha$, therefore $_\alpha \mu _A (x)=\alpha$. On the other hand, for each $\alpha \in (a,1]$, we have $\mu_A (x) =a < \alpha$ and $_\alpha \mu _A (x)=0$.\\
Similarly, for each $\beta \in [0,b]$, we have $\nu _A (x)=b \geq \beta$, therefore $_\beta \nu _A (x)=\beta$. On the other hand, for each $\beta \in (b,1]$, we have $\nu_A (x) =b < \beta$ and $_\beta \nu _A (x)=0$. Therefore $\left( \left({\bigcup _{\alpha  \in [0,1]}} ~~ _{\alpha } \mu_A \right) (x) , \left({\bigcup_{\beta  \in [0,1]}} ~~ _{\beta } \nu_A\right)(x)\right)$\\=$\left( Sup _{\alpha  \in [0,a]} ~~  \alpha ,Sup_{\beta  \in [0,b]} ~~ \beta \right)$= $(a,b)=$\\$\left(\mu_A (x), \nu_A (x)\right)$. Hence the theorem.
\end{proof}
To illustrate the above theorem, let us consider a trapezoidal intuitionistic fuzzy number  $A$ as in figure \ref{c1}.

For each $\alpha, \beta \in [0,1]$, the $\alpha - \beta$ cut of $A=\{(a, b_1, b_2, c), (e, f_1, f_2, g) \}$  is given by
$^{(\alpha - \beta)}A =([a+(b_1 -a)\alpha, c-(c-b_2)\alpha],[e+(f_1 -e )\beta, g-(g-f_2)\beta])$ and the special intuitionistic fuzzy set $_{(\alpha,\beta)}{A}$ employed in definition \ref{dfen1} is defined by its membership ($_{\alpha}{\mu_A}$) and non-membership function ($_{\beta}{\nu_A}$) as follows\\
$_{\alpha}{\mu_A}(x)=\begin{cases} \text\ \alpha ,~ x \in [a+(b_1 -a)\alpha, c-(c-b_2)\alpha]
\text\ \ \\~ \  \\\ 0,~otherwise \end{cases} 
$ \\
$_{\beta}{\nu_A}(x)=\begin{cases} \text\  \beta ,~ x \in [e+(f_1 -e )\beta, g-(g-f_2)\beta]
\text\ \ \\~ \  \\\ 0 ,~otherwise \end{cases} 
$ \\

Examples of sets $^{\alpha}{\mu_A}$, $^{\beta}{\nu_A}$, $_{\alpha}{\mu_A}$ and $_{\beta}{\nu_A}$ for three values of $\alpha$ (namely $\alpha_1,\alpha_2,\alpha_3$) and $\beta$ (namely $\beta_1,\beta_2,\beta_3$) are shown in figure \ref{c1}. 
\begin{figure}[h!]
\centering
\scalebox{0.8}{
 \resizebox{7.5cm}{8cm}{\includegraphics{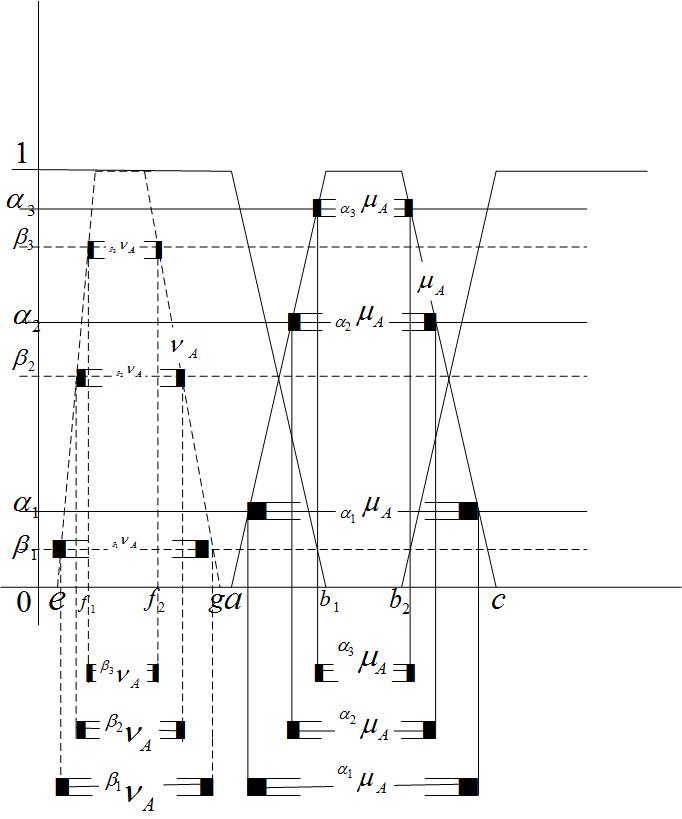}}
}
\caption {Illustration of First Decomposition Theorem}
\label{c1}
\end{figure} \\According to theorem \ref{CS1}, $A$ is obtained by taking the standard fuzzy union of sets ($_{\alpha}{\mu_A}$,  $_{\beta}{\nu_A}$) for all $\alpha , \beta \in [0,1]$.
\begin{thm} 
{\bf Second Decomposition theorem of an IFS}:
Let $X$ be a non-empty set. For an intuitionistic fuzzy subset $A$ in $X$, \\$A=\left({\bigcup _{\alpha  \in [0,1]}} ~~ _{\alpha+ } \mu_A , {\bigcup_{\beta  \in [0,1]}} ~~ _{\beta+ } \nu_A\right)$,  where $\bigcup $ is standard union.
\label{SDT2}
\end{thm}
\begin{proof}
Let  $x$ be an arbitrary element in $ X$ and  let $\mu_A (x)=a~\&~ \nu_A (x)=b$.\\ Then $\left( \left({\bigcup _{\alpha  \in [0,1]}} _{\alpha+ } \mu_A \right) (x) , \left({\bigcup_{\beta  \in [0,1]}} _{\beta +} \nu_A\right)(x)\right)=\left( Sup_{\alpha  \in [0,1]} ~~ _{\alpha +} \mu_A  (x) , Sup_{\beta  \in [0,1]} ~~ _{\beta +} \nu_A(x)\right)$\\=$( max \left[Sup_{\alpha  \in [0,a)} ~ _{\alpha +} \mu_A (x) ,Sup_{\alpha  \in [a,1]} ~ _{\alpha +} \mu_A (x)\right], \\ max \left[ Sup_{\beta  \in [0,b)} ~~ _{\beta+ } \nu_A (x) , Sup_{\beta  \in [b,1]} ~~ _{\beta +} \nu_A (x) \right] )$.
\\For each $\alpha \in [0,a)$, we have $\mu _A (x)=a > \alpha$, therefore $_{\alpha +} \mu _A (x)=\alpha$. On the other hand, for each $\alpha \in [a,1]$, we have $\mu_A (x) =a \leq \alpha$ and $_{\alpha + } \mu _A (x)=0$.\\
Similarly, for each$\beta \in [0,b)$, we have $\nu _A (x)=b > \beta$, therefore $_{\beta +} \nu _A (x)=\beta$. On the other hand, for each $\beta \in [b,1]$, we have $\nu_A (x) =b \leq \beta$ and $_{\beta +} \nu _A (x)=0$. Therefore \\$\left( \left({\bigcup _{\alpha  \in [0,1]}} ~~ _{\alpha + } \mu_A \right) (x) , \left({\bigcup_{\beta  \in [0,1]}} ~~ _{\beta +} \nu_A \right)(x)\right)$=\\$\left( Sup _{\alpha  \in [0,a)} \alpha ,Sup_{\beta  \in [0,b)}  \beta \right)$= $(a,b)=\left(\mu_A (x), \nu_A (x)\right)$. Hence the theorem.
\end{proof}

\begin{thm}
{\bf Third Decomposition theorem of an IFS}:
Let $X$ be a non-empty set. For an intuitionistic fuzzy subset $A$ in $X$,\\ $A=\left({\bigcup _{\alpha  \in L_{\mu_A}}} ~~ _{\alpha } \mu_A , {\bigcup_{\beta  \in L_{\nu_A}} }~~ _{\beta } \nu_A\right)$, where $\bigcup $ is standard union .
\end{thm}

\begin{proof}
The proof of this theorem is similar to the Theorem \ref{SDT2}. 
\end{proof}
$~~~~~$Regarding IFN as special intuitionistic fuzzy subset (IFS) of $\Re$, these decomposition theorems are also available for IFNs. Since they identify an intuitionistic fuzzy number by uncountably infinite real valued parameters generally and therefore\\ lexicography can not be used anymore, unfortunately none of them can be used to define a total ordering on the set of IFNs. Thus, establising a new decomposition theorem, which identifies any IFN by only countably many real valued parameters is essential.
\begin{thm} \textbf{Fourth Decomposition theorem of an IFS:}
Let $A=\left(\mu_A,\nu_A\right)$ be an intuitionistic fuzzy subset of $X$, and $S=(S_1,S_2)$ be a given double upper dense sequence in $[0,1]$. \\Then $A=\left({\bigcup _{\alpha \in S_1}} ~~ _{\alpha} {\mu_A} , {\bigcup _{\beta \in S_2}} ~~ _{\beta} {\nu_A}\right)$.
\end{thm}

\begin{proof}
Let  $x$ be an arbitrary element in $ X$ and  let $\mu_A (x)=a~\&~ \nu_A (x)=b$.
Since $S_1 \subseteq [0,1] ~\&~ S_2 \subseteq [0,1]$, Thus we have $\bigcup_{\alpha \in S_1}~~ _{\alpha}{\mu_A} \subseteq {\bigcup _{\alpha \in [0,1]}} ~~ _{\alpha} {\mu_A}=\mu_A$ and $ {\bigcup _{\beta \in S_2}} ~~ _{\beta} {\nu_A} \subseteq {\bigcup _{\beta \in [0,1]}} ~~ _{\beta} {\nu_A}=\nu_A $ .\\
So $\left({\bigcup _{\alpha \in S_1}} ~~ _{\alpha} {\mu_A} , {\bigcup _{\beta \in S_2}} ~~ _{\beta} {\nu_A}\right) \subseteq A$ ~~~~~~~~~~~ ........(1) 
\\Now we have to show that $\mu_A(x) \leq {\bigcup _{\alpha \in S_1}} ~~ _{\alpha} {\mu_A}(x) $ and $ \nu_A(x) \leq {\bigcup _{\beta \in S_2}} ~~ _{\beta} {\nu_A}(x)$ to prove the theorem.\\
From second decomposition theorem we know that , for each $x \in X$, $\left(\mu_A(x),\nu_A(x)\right)=\\\left( Sup_{\alpha \in [0,1]}~~{_{\alpha+}{\mu_A}}(x), Sup_{\beta \in [0,1]}~~{_{\beta+}{\nu_A}}(x)\right)=\\ \left( Sup_{\alpha \in [0,a)}~~{_{\alpha+}{\mu_A}}(x), Sup_{\beta \in [0,b)}~~{_{\beta+}{\nu_A}}(x)\right)$.\\For each $\alpha \in [0,a)$, since $S_1$ is upper dense in $(0,1]$. We may find a real number ${K_1} \in S_1$ such that $K_1 \geq \alpha$ , which implies that $K_1 \in [\alpha, a)$ and  ${_{\alpha+}{\mu_A}}(x) < {_{{K_1}}{\mu_A}}(x) \leq Sup _ {{K_1} \in S_1} ~~ _{{K_1}}{\mu_A}(x)$. Thus taking the supremum with respect to $\alpha \in (0, \mu_A(x))$, we obtain $\mu_A (x)=Sup_{\alpha \in (0,a)}~{_{\alpha+}{\mu_A}}(x) \leq Sup _ {{K_1} \in S_1} ~ _{{K_1}}{\mu_A}(x)\\=Sup _ {{\alpha} \in S_1} ~ _{{\alpha}}{\mu_A}(x)$. i.e., $\mu_A(x) \leq {\bigcup _{\alpha \in S_1}} ~~ _{\alpha} {\mu_A}(x)$.\\
Similarly for each $\beta \in [0,b)$, since $S_2$ is upper dense in $[0,1]$. We may find a real number ${K_2} \in S_2$ such that  $K_2 \geq \beta$ , which implies that $K_2 \in [\alpha, b)$ and  ${_{\beta+}{\nu_A}}(x) < {_{{K_2}}{\nu_A}}(x) \leq Sup _ {{K_2} \in S_2} ~~ _{K_2}{\nu_A}(x)$. Thus taking the supremum with respect to $\beta \in (0,\nu_A(x))$, we obtain $\nu_A(x)= Sup_{\beta \in (0,b)} ~~ {_{\beta+}{\nu_A}}(x) \leq \\ Sup _ {{K_2} \in S_2} ~~ _{K_2}{\nu_A}(x)=Sup _ {{\beta} \in S_2}~ _{\beta}{\nu_A}(x) $.\\i.e., $\nu_A(x) \leq {\bigcup _{\beta \in S_2}} ~~ _{\beta} {\nu_A}(x)$.\\ Hence $A(x)=\left(\mu_A(x),\nu_A(x)\right) \subseteq \\ \left({\bigcup _{\alpha \in S_1}} ~ _{\alpha} {\mu_A}(x) , {\bigcup _{\beta \in S_2}} ~ _{\beta} {\nu_A}(x)\right)$.~~~~~~~~~~~~~~~~~~~......(2)\\
(1) and (2) concludes the proof.
\end{proof}
\section{Total Ordering on the set of all Intuitionistic Fuzzy Numbers}
 $~~~~~$The new decomposition theorem established in section 3 identifies an arbitrary intuitionistic fuzzy number by a countably many real-valued parameters. It provides us with a powerful tool for defining total order in the class of IFN, by extending lexicographic ranking relation defined in Geetha et.al \cite{9}.\\
\begin{nt} Let $A $ be an intuitionistic fuzzy number. Let  $S= \left\{ (\alpha_i, \beta_i) | i=1,2,\ldots\right\} \in [0,1]$ be an double upper dense sequence. The $(\alpha_i-\beta_i)$-cut of an IFN $A$ at each $\alpha _ i , \beta_i$ , $i=1,2,...$, is a combination of two closed intervals. Denote this intervals by $([a_i, b_i],[c_i,d_i])$, where $[a_i,b_i]$ is the $\alpha$-cut of the membership function of $A$ and $[c_i,d_i]$ is the $\beta$-cut of the non-membership function of $A$ and \\let $C_{4i-3}= \frac{a_i + b_i - c_i - d_i + a_ic_i + b_id_i}{2}$,\\$C_{4i-2}= \frac{a_i + b_i -c_i - d_i - a_ic_i - b_id_i}{2}$ ,\\$C_{4i-1}= \frac{a_i - b_i - c_i + d_i + a_ic_i + b_id_i}{2}$ ,\\$C_{4i}= \frac{ a_i - b_i + c_i - d_i + a_ic_i - b_id_i}{2}, i=1,2,\ldots $. \label{nt 1}
\end{nt} By fourth decomposition theorem these countably many parameters $\left\{ C_j | j=1,2, \ldots \right\}$ identify the \\intuitionistic fuzzy number. Using these parameters, we define a relation on the set of all intuitionistic fuzzy number as follows.

\begin{defn}(Ranking Principle)
Let $A~and~B$ be any two IFNs. Consider any double upper dense sequence $S=\left\{ (\alpha_i, \beta_i) | i=1,2,\ldots \right\}$ in $[0,1]$, using note \ref{nt 1}, for each $i=1,2...$ we have $C_{4i-3},C_{4i-2},$\\$C_{4i-1},C_{4i}$, which describe a sequences of $C_j(A)$ and $C_j(B)$. If $A \neq B$, then there exists $j$ such that $C_j(A) \neq C_j(B)$ and $C_k(A)=C_k(B)$ for all positive integers $k < j$. The '$<$' relation on the set of of all intuitionistic fuzzy number (IFN) is defined as,\\ $A < B$ if there exists $j$ such that $C_j(A) < C_j(B)$ and $C_k(A)=C_k(B)$ for all positive integers $k < j$. 
\label{rank}
\end{defn}
The above ranking principle on IFNs is  illustrated in the following examples. 
\begin{ex} 
Let\\ $A=\langle (0.20,0.30,0.50),(0.35,0.55,0.65)\rangle$ \\$B=\left\langle (0.17,0.32,0.58),(0.37,0.63,0.73)\right\rangle$ and\\ $C=\left\langle (0.25,0.40,0.70),(0.45,0.75,0.85)\right\rangle$ be three triangular intuitionistic fuzzy numbers (TIFN). \\
The ordering $<$ defined by using double upper dense sequence $S$ given in example \ref{eg} and the way shown in definition \ref{rank} are now adapted. Thus we have ${^{(\alpha - \beta)}}A =([0.2+0.1\alpha, 0.5-0.2\alpha],[0.35+0.2\beta, 0.65-0.1\beta])$, ${^{(\alpha - \beta)}}B=([0.17+0.15\alpha, 0.58-0.26\alpha],[0.37+.26\beta,0.73-0.1\beta])$,${^{(\alpha - \beta)}}C=([0.25+0.15\alpha, 0.70-0.3\alpha],[0.45+0.3\beta, 0.85-0.1\beta])$. For $i=1, (\alpha_1, \beta_1)=(1,1)$, $C_1(A)=-0.85,C_1(B)=-0.1084, C_1(C)=-0.05 $. i.e., $C_1(A)<C_1(B)<C_1(C)$. Hence $A<B<C$.
\end{ex}
\begin{ex}
Let \\$A=\left\langle (0.35,0.35,0.4,0.6),(0.1,0.2,0.3,0.35)\right\rangle$ \\$B=\left\langle (0.35,0.35,0.45,0.55),(0,0.3,0.3,0.35)\right\rangle$ be two TrIFNs. \\
Thel ordering $<$ defined by  using double upper dense sequence $S$ given in example \ref{eg} and the way shown in definition \ref{rank} are now adapted. We have ${^{(\alpha - \beta)}}A =([0.35, 0.6-0.2\alpha],[0.1+0.1\beta, 0.35-0.05\beta])$, ${^{(\alpha - \beta)}} B=([0.35, 0.55-0.1\alpha],[0.3\beta,0.35-0.05\beta])$. For $i=1, (\alpha_1, \beta_1)=(1,1)$,  $C_1(A)=0.22=C_1(B)$, $C_2(A)=-0.03, C_2(B)=0.02$. 
Since $C_2(A)<C_2(B)$. Hence $A>B$. From this example we come to know that $C_1$ alone can not rank any two given IFNs. With the help of $C_2$ we discriminate $A$ and $B$. 
\end{ex}
The following example shows the importance of double upper dense sequence. In the previous examples, different IFNs are ranked by means of $C_1$ and $C_2$. But in general $C_1~and~C_2$ alone need not be enough to rank  the entire class of intuitionistic fuzzy number which is shown in example \ref{exdg}.
\begin{ex}
Let \\$A=\left\langle (0.3,0.35,0.4,0.5),(0.1,0.2,0.25,0.3)\right\rangle$ \\$B=\left\langle (0.35,0.35,0.4,0.55),(0,0.2,0.25,0.35)\right\rangle$ be two TrIFNs. \\
The ordering $<$ defined by  using double upper dense sequence $S$ given in example \ref{eg} and the way shown in definition \ref{rank} are now adapted. We have ${^{(\alpha - \beta)}}A =([0.3+0.05\alpha, 0.5-0.1\alpha],[0.1+0.1\beta, 0.3-0.05\beta])$, ${^{(\alpha - \beta)}} B=([0.35, 0.55-0.15\alpha],[0.2\beta,0.35-0.1\beta])$. For $i=1, (\alpha_1, \beta_1)=(1,1)$, we have $C_1(A)=0.235=C_1(B)$, $C_2(A)=0.065= C_2(B)$, $C_3(A)=0.085= C_3(B)$, $C_4(A)=-0.065= C_4(B)$. Now we have to find for $i=2,(\alpha_2, \beta_2)=(1/2,1/2)$, then $C_5(A)=0.26125 < C_5(B)=0.30125$.
Since $C_5(A)<C_5(B)$. Hence $A<B$.
\label{exdg}
\end{ex}
\begin{thm}
Relation $<$ is a total ordering on the set of all intuitionistic fuzzy number.
\end{thm}
\begin{proof}
Claim: $<$ is total ordering on the set of IFN.\\ To prove $<$ is total ordering we need to show the following (1). $<$ is Partial ordering on the set of IFN\\ (2). Any two elements of the set of IFNs are comparable.\\
(1). $<$ is partial ordering on the set of IFN:\\
(i) $<$ is refelxive:\\
It is very clear that the relation $<$ is reflexive for any $A$.\\
(ii) $<$ is antisymmetric:\\ 
claim: If $A < B$ and $B < A$ then $A = B$.\\
Suppose $A \neq B$, then from the hypothesis $A \prec B$ and $B \prec A$. From the definition \ref{rank}, we can find $j_1$ such that $C_{j_1} (A) < C_{j_1}(B)$ and $C_{j}(A)=C_{j}(B)$ for all positive integers $j<j_1$. Similarly we are able to find $j_2$ such that $C_{j_2} (A) < C_{j_2}(B)$ and $C_{j}(A)=C_{j}(B)$ for all positive integers $j<j_2$. Then $j_1 \& j_2 $ must be the same, let it to be $j_0$. But $C_{j_0}(A)<C_{j_0}(B)$, and $C_{j_0}(B)<C_{j_0}(A)$ this contradicts our hypothesis. Therefore our assumption $A \neq B$ is wrong. Hence $A=B$.\\
(iii) $<$ is transitive:\\
To prove (iii), we have to show that if $A < B$ and $B < C$ then $A < C$\\
Let $A,B,C $ be three IFNs. Let us assume $A < B$ and $B < C$. ....(1)\\
Therefore from $A < B$, we can find a positive integer $k_1$ such that $C_{k_1}(A)< C_{k_1}(B)$ and $C_{k}(A) = C_{k}(B)$ for all positive integer $k<k_1$. Similarly from $B < C$, we can find a positive integer $k_2$ such that $C_{k_2}(B)< C_{k_2}(C)$ and $C_{k}(B) = C_{k}(C)$ for all positive integer $k<k_2$. Now taking $j_0 = min (k_1 , k_2)$, we have $C_{k_0}(A)< C_{k_0}(C)$ and $C_{k}(A) = C_{k}(C)$ for all positive integer $k<k_0$. i.e., $A < C$. Hence $<$ is Transitive.\\ Therefore from (i), (ii), and (iii), we proved the relation $<$ is Partial Ordering on the set of all IFNs. \\
(2). Any two elements of the set of IFNs are comparable.\\
For any two IFNs $A$ and $B$, they are either $A=B$, or $A \neq B$. In the latter case, there are some integers $j$ such that $C_j(A) \neq C_j(B)$. Let $J= \left\{ j | C_j(A) \neq C_j(B) \right\}$. Then $J$ is lower bounded 0 and therefore, according to the well ordering property , $J$ has unique smallest element, denoted by $j_0$. Thus we have $C_j(A)=C_j(B)$ for all positive integers $j< j_0$, and either $C_{j_0}(A) < C_{j_0}(B)$ or $C_{j_0}(A) > C_{j_0}(B)$, that is, either $A \prec B$ or $B \prec A$ in this case. So, for these two IFNs, either $A < B$ or $B < A$. This means that partial ordering $<$ is a total ordering on the set of all intuitionistic fuzzy numbers.
 Hence the proof.
\end{proof}

Similar to the case of total orderings on the real line $(- \infty, + \infty)$ an the total ordering on the sets of special types of IFNs shown in section 3, infinitely many different total orderings on the set of IFNs can be defined. Even using a given upper dense sequence in $[0,1]$, there are still infinitely many different ways to determine a total ordering on the set of IFNs. A notable fact is that each of them is consistent with the natural ordering on the set of all real numbers. This can be regarded as a fundamental requirement for any practice ordering method on the set of all IFNs.

\section{Significance of the proposed method}

$~~~~~$Many researchers have proposed different ranking methods on IFNs, but none of them has covered the entire class of IFNs, and also almost all the methods have disadvantage that at some point of time they ranked two different numbers as the same. In this paper a special type of IFNs which is shown in figure 1 which generalizes IFN more natural in real scenario. Problems in different fields involving qualitative, quantitative and uncertain information can be modelled better using this type of IFNs when compared with usual IFNs.  Our proposed ranking method on this type of IFN will give the better results over other existing methods, and this paper will give the better understanding over this new type of IFNs. This type of IFNs are very much important in real life problems and this paper will give the significant change in the literature. Modeling problems using this type of IFN will give better result. In this subsection our proposed method is compared with the total score function defined in Lakshmana et al.\cite{13}, which is explained here with an illustrative example.

\subsection{Comparision between our proposed method with the score function defined in Lakshmana et al.\cite{13}:}
$~~~~~$In this subsection, our proposed method is compared with the total score function defined in Lakshmana et al.\cite{13}, which is explained here with an illustrative example.
\begin{defn}(Note 1.2\cite{33})\\
The membership score of the triangular intuitionistic fuzzy number (TIFN) $M= \left\{(a,b,c)(e,f,g)\right\}$ is defined by $T(M)=\frac{[1+R(M)-L(M)]}{2}$, where $L(M)=\frac{1-a}{1+b-a}$ and $R(M)=\frac{c}{1+c-b}$.
\label{Nx1}
\end{defn}
\begin{defn}(Lakshmana et al.\cite{13})\\
The new membership score of the triangular \\intuitionistic fuzzy number $M= \left\{(a,b,c)(e,f,g)\right\}$ is defined by $NT(M)=\frac{[1+NL(M)-NR(M)]}{2}$, where $NL(M)=\frac{e}{1+e-f}$ and $NR(M)=\frac{1-g}{1+f-g}$.
\label{nx2}
\end{defn}
\begin{defn}(Lakshmana et al.\cite{13})\\
The nonmembership score of the triangular \\intuitionistic fuzzy number $M= \left\{(a,b,c)(e,f,g)\right\}$ is defined by $NT_c(M)=1-NT(M)$.
\label{nx3}
\end{defn}
\begin{defn}(Lakshmana et al.\cite{13})\\
Let $M=\left\langle (a,b,c),(e,f,g)\right\rangle$ be an triangular \\intuitionistic fuzzy number. If $e \geq b ~and~ f \geq c$, then the score of the intuitionistic fuzzy number $M$ is defined by $(T,NT_c)$, where $T$ is the membership score of $M$ which is obtained from $(a,b,c)$ and $NT_c$ is the nonmembership score of $M$ obtained from $(e,f,g)$.
\label{NTc}
\end{defn}

\begin{defn}(Lakshmana et al.\cite{13})\\ Ranking of Intuitionistic Fuzzy Numbers:\\
Let $M_1=\left\langle (a_1,b_1,c_1),(e_1,f_1,g_1)\right\rangle$ and $M_2=\left\langle (a_2,b_2,c_2),(e_2,f_2,g_2)\right\rangle$ be two triangular intuitionisti\\c fuzzy numbers with $e_i \geq b_i ~\&~ f_i \geq c_i$. Then $M_1 \leq M_2$ if the membership score of $M_1 \leq$ the membership score of $M_2$ and the nonmembership score of $M_1 \geq$ the nonmembership score of $M_2$.
\label{eg4}
\end{defn}
In this example definition \ref{Nx1} to \ref{eg4} are demonstr\\ated and also the illogicality of Lakshmana etal's \cite{13} method is shown.
\begin{ex} Let $M=\left\langle (0,0.2,0.4),(0.4,0.45,.5)\right\rangle$ with ${(0,0.2,0.4)}^c \leq (0.4,0.45,0.5)$ and $0.4 \geq 0.2 ~ \& ~ 0.45 \geq 0.4 $ and $N=\langle (0.25,0.25,0.25),(0.4,\\0.45,0.5)\rangle$ with ${(0.25,0.25,0.25)}^c \geq (0.4,0.45,0.5)$ and $0.4 \geq 0.25 ~ \& ~ 0.45 \geq 0.25$ be the two triangular intuitionistic fuzzy numbers. Using definition \ref{Nx1} we get $L(M)=0.8333$ and $R(M)=0.3333$ which implies $T(M)=0.25$. Applying definition \ref{nx2}, \ref{nx3} \ref{NTc} to $M$ we get $NL(M)=0.421053, NR(M)=0.526316$, $NT(M)=0.447368$ and $NT_c(M)=0.553$. Therefore the total score of membership and nonmembership functions of $M$ are $T=0.25$ and $ NT_c=0.553$ respectively. Then the intuitionistic fuzzy score is represented by $(T,NT_c)=(0.25,0.553)$. Similarly using definitions \ref {Nx1} to \ref{nx3}, we get the total score of membership and nonmembership functions of $N$ are $T=0.25, ~and ~ NT_c=0.553$ respectively.

Therefore from definition \ref{eg4}, this method ranks $M$ and $N$ are equal but $M ~\&~ N$ are different triangular intuitionistic fuzzy numbers.\\
The total ordering $<$ defined by using double upper dense sequence $D_{(\alpha, \beta)}$ given in example \ref{eg} and the way shown in definition \ref{rank} are now adapted. For $i=1 ~and~ (\alpha_1, \beta_1)=(1,1)$, we have $C_1(M)=-0.16 ~and~ C_1(N)=-0.0875$.\\ i.e., $C_1(M)<C_1(N)$. Hence our proposed method ranks $N$ as better one.
\end{ex}
\subsection{Comparision of the proposed method with some existing methods}
$~~~~~$In this sub section significance of our proposed method over some existing methods are explained with examples.   
The table \ref{ta1} shows that our proposed method is significant over the methods presented in \cite{Ch,8,Dh,123,11,Lin,Ll,Li,Jw,25,Ju,Z,Hz} , and it is supported by Geetha et.al \cite{9}. \\For example, let $A=([0.2,0.25],[0.4,0.45])$ and $B=([0.15,0.3],[0.35,0.5])$ be two IVIFNs. Then by applying Lakshmana and Geetha \cite{11} approach we get $LG(A)=LG(B) =\frac{0.45+ \delta (0.70)}{2}$ this implies that $A$ and $B$ are equal which is illogical. By applying the total ordering $<$ defined by using double upper dense sequence $D_{(\alpha, \beta)}$ given in example \ref{eg} and the way shown in definition \ref{rank} are now adapted. For $i=1 ~and~ (\alpha_1, \beta_1)=(1,1)$, we have $C_1(A)=-0.10375 ~and~ C_1(B)=-0.09875$, i.e., $C_1(A)<C_1(B)$. Hence $A<B$ which is supported by Geetha et.al's approach.
 The double upper dense sequence $S$ in example \ref{eg} is used as a necessary reference system for our proposed method in   table 1  (i.e., for $i=1, (\alpha_1,\beta_1)=(1,1)$).

\begin{table*}[htb]
\caption{Significance of proposed method}
\label{ta1}
\centering 
\small
\scalebox{0.53}{
\begin{tabular}{l l l l l} 
\hline
Other exisiting Methods & shortcomings of exisiting methods & Numerical example & Geetha et.al \cite{9} & Proposed Method  \\ \\ \hline 
Xu, Z.S. \cite{25} \\
 $s(A)=\frac{a+b-c-d}{2}$ &  $A=([a_1,b_1],[c_1, d_1])$ , & $A=([0,0.3],[0.35,0.65])$  & $L(A)=-0.2525$ &$C_1(A)=-0.2525,$\\ $h(A)=\frac{a+b+c+d}{2}$ & $B=([a_1- \epsilon,b_1+\epsilon],[c_1-\epsilon, d_1+\epsilon])$ & $ B=([0.1,0.2],[0.45,0.55])$ & L(B)=-0.2725& $C_1(B)=-0.2725$\\    & $s(A)=s(B)=\frac{a_1+b_1-c_1-d_1}{2} ,$ &$s(A)=s(B) =-0.35,$ & $L(A)>L(B) \Rightarrow A>B$ &$ C_1(A)>C_1(B) \Rightarrow A>B$ \\ & $ h(A)=h(B)=\frac{a_1+b_1+c_1+d_1}{2} \Rightarrow A=B$ & $ h(A)=h(B)=0.65 \Rightarrow A=B$ & &  \\ \hline
Dejian Yu,et.al., \cite{8} \\
$S(A)=\frac{2+a+b-c-d}{2}$ &  $A=([a_1,b_1],[c_1, d_1])$ , & $A=([0,0.3],[0.35,0.65])$  & $L(A)=-0.2525$ &$C_1(A)=-0.2525,$\\  & $B=([a_1- \epsilon,b_1+\epsilon],[c_1-\epsilon, d_1+\epsilon])$ & $ B=([0.1,0.2],[0.45,0.55])$ & L(B)=-0.2725& $C_1(B)=-0.2725$\\    & $S(A)=S(B)=\frac{2+a_1+b_1-c_1-d_1}{2} ,$ &$S(A)=S(B) =0.65,$ & $L(A)>L(B) \Rightarrow A>B$ &$ C_1(A)>C_1(B) \Rightarrow A>B$ \\ & $ \Rightarrow A=B$ & $ \Rightarrow A=B$ & &  \\ \hline
	Jun Ye, \cite{123}\\
$M(A)=a+b-1+\frac{c+d}{2}$ &  $A=([a_1,b_1],[c_1, d_1])$ , & $A=([0.1,0.15],[0.25,0.35])$  & $L(A)=-0.13625$ &$C_1(A)=-0.13625,$\\  & $B=([a_1- \epsilon,b_1+\epsilon],[c_1-\epsilon, d_1+\epsilon])$ & $ B=([0.05,0.2],[0.20,0.40])$ & L(B)=-0.13& $C_1(B)=-0.13$\\    & $M(A)=M(B)=a_1+b_1-1+\frac{c_1+d_1}{2} ,$ &$M(A)=M(B) =-0.45,$ & $L(B)>L(A) \Rightarrow B>A$ &$ C_1(B)>C_1(A) \Rightarrow B>A$ \\ & $ \Rightarrow A=B$ & $ \Rightarrow A=B$ & &  \\ \hline
Lakshmana and Geetha \cite{11}\\
$LG(A)=\frac{a+b+ \delta (2-a-b-c-d)}{2}, \forall \delta \in [0,1]$ &  $A=([a_1,b_1],[c_1, d_1])$ , & $A=([0.2,0.25],[0.40,0.45])$  & $L(A)=-0.10375$ &$C_1(A)=-0.10375,$\\  & $B=([a_1- \epsilon,b_1+\epsilon],[c_1-\epsilon, d_1+\epsilon])$ & $ B=([0.15,0.30],[0.35,0.50])$ & L(B)=-0.09875& $C_1(B)=-0.09875$\\    & $LG(A)=LG(B)=\frac{a_1+b_1+ \delta (2-a_1-b_1-c_1-d_1)}{2} ,$ &$LG(A)=LG(B) =\frac{0.45+ \delta (0.70)}{2},$ & $L(B)>L(A) \Rightarrow B>A$ &$ C_1(B)>C_1(A) \Rightarrow B>A$ \\ & $ \Rightarrow A=B$ & $ \Rightarrow A=B$ & &  \\ \hline
Chen and Tan \cite{Ch}\\
$S(A)=\mu-\nu,$ &  $A=(a_1,b_1),B=(a_1-\epsilon, b_1-\epsilon), \forall \epsilon \in (0,1]$ & $A=(0.6,0.2),B=(0.47,0.07)$ & $L(A)=0.52, L(B)=0.4329$ & $C_1(A)=0.52, C_1(B)=0.4329$ \\ 
$~where~A=(\mu,\nu)~with~ \mu+\nu \leq 1$ & $S(A)=S(B)=a_1-b_1 \Rightarrow A=B$ & $S(A)=S(B)=0.4 \Rightarrow A=B$ & $L(A)>L(B) \Rightarrow A>B$ & $C_1(A)>C_1(B) \Rightarrow A>B$ \\ \hline
Hong and Choi \cite{Dh}\\
$H(A)=\mu+\nu,$ &  $A=(a_1,b_1),B=(a_1-\epsilon, b_1+\epsilon), \forall \epsilon \in (0,1]$ & $A=(0.4,0.3),B=(0.6,0.1)$ & $L(A)=0.22, L(B)=0.56$ & $C_1(A)=0.22, C_1(B)=0.56$ \\ 
$~where~A=(\mu,\nu)~with~ \mu+\nu \leq 1$ & $H(A)=H(B)=a_1+b_1 \Rightarrow A=B$ & $H(A)=H(B)=0.7 \Rightarrow A=B$ & $L(A)<L(B) \Rightarrow A<B$ & $C_1(A)<C_1(B) \Rightarrow A<B$ \\ \hline
Liu \cite{Li}\\
$S_{Li}(A)=\mu+\mu(1-\mu-\nu),$ &  $A=(0,b_1),B=(0, b_2), ~where~ b_1<b_2$ & $A=(0,0.1),B=(0,0.7)$ & $L(A)=-0.1, L(B)=-0.7$ & $C_1(A)=-0.1, C_1(B)=-0.7$ \\ 
$~where~A=(\mu,\nu)~with~ \mu+\nu \leq 1$ & $S_Li(A)=S_Li(B)=0 \Rightarrow A=B$ & $S_{Li}(A)=S_Li(B)=0 \Rightarrow A=B$ & $L(A)>L(B) \Rightarrow A>B$ & $C_1(A)>C_1(B) \Rightarrow A>B$ \\ \hline
Zhou and Wu \cite{Z}\\
$S_{Zh}(A)=\mu-\nu+(\alpha-\beta)(1-\mu-\nu),$ &  $A=(a_1,b_1),B=(a_1-\epsilon, b_1-\epsilon), \forall \epsilon \in (0,1] \& \alpha=\beta$ & $A=(0.6,0.1),B=(0.5,0)$ & $L(A)=0.56, L(B)=0.5$ & $C_1(A)=0.56, C_1(B)=0.5$ \\ 
$~where~A=(\mu,\nu), \alpha,\beta \in[0,1]~with~ \alpha+\beta \leq 1$ & $S_{Zh}(A)=S_{Zh}(B)=a_1-b_1 \Rightarrow A=B$ & $S_{Zh}(A)=S_{Zh}(B)=0.5 \Rightarrow A=B$ & $L(A)>L(B) \Rightarrow A>B$ & $C_1(A)>C_1(B) \Rightarrow A>B$ \\ \hline
Lin etal \cite{Lin}\\
$S_{Lin}(A)=2\mu+\nu-1,$ &  $A=(a_1,b_1),B=(a_1+\frac{\epsilon}{2}, b_1-\epsilon), \forall \epsilon \in (0,1] $ & $A=(0.25,0.1),B=(0.3,0)$ & $L(A)=0.175, L(B)=0.3$ & $C_1(A)=0.175, C_1(B)=0.3$ \\ 
$~where~A=(\mu,\nu)~with~ \mu+\nu \leq 1$ & $S_{Lin}(A)=S_{Lin}(B)=2a_1+b_1-1 \Rightarrow A=B$ & $S_{Lin}(A)=S_{Lin}(B)=-0.4 \Rightarrow A=B$ & $L(A)<L(B) \Rightarrow A<B$ & $C_1(A)<C_1(B) \Rightarrow A<B$ \\ \hline
Wang etal \cite{Jw}\\
$S_{W}(A)=\mu-\nu-\frac{(1-\mu-\nu)}{2},$ &  $A=(a_1,b_1),B=(a_1+\frac{\epsilon}{3}, b_1+\epsilon), \forall \epsilon \in (0,1] $ & $A=(0.34,0.2),B=(0.44,0.5)$ & $L(A)=0.208, L(B)=0.16$ & $C_1(A)=0.208, C_1(B)=0.16$ \\ 
$~where~A=(\mu,\nu)~with~ \mu+\nu \leq 1$ & $S_{W}(A)=S_{W}(B)=\frac{3a_1-b_1-1}{2} \Rightarrow A=B$ & $S_{W}(A)=S_{W}(B)=-0.09 \Rightarrow A=B$ & $L(A)>L(B) \Rightarrow A>B$ & $C_1(A)>C_1(B) \Rightarrow A>B$ \\ \hline
L.Lin etal \cite{Ll}\\
$S_{L}(A)=\frac{\mu}{2}+\frac{3}{2}(\mu+\nu)-1,$ &  $A=(a_1,b_1),B=(a_1+\frac{\epsilon}{2}, b_1-\frac{2\epsilon}{3}), \forall \epsilon \in (0,1]$ & $A=(0.52,0.31),B=(0.62,0.377)$ & $L(A)=0.3712, L(B)=0.47674$ & $C_1(A)=0.3712, C_1(B)=0.47674$ \\ 
$~where~A=(\mu,\nu)~with~ \mu+\nu \leq 1$ & $S_{L}(A)=S_{L}(B)=\frac{4a_1+3b_1-2}{2} \Rightarrow A=B$ & $S_{L}(A)=S_{L}(B)=0.505 \Rightarrow A=B$ & $L(A)<L(B) \Rightarrow A<B$ & $C_1(A)<C_1(B) \Rightarrow A<B$ \\ \hline
Ye \cite{Ju}\\
$S_{Y}(A)=\mu(2-\mu-\nu)+(1-\mu-\nu)^2,$ & -  & $A=(0.3,0.6),B=(0.3,0.5)$ & $L(A)=-0.12, L(B)=-0.05$ & $C_1(A)=-0.12, C_1(B)=-0.05$ \\ 
$~where~A=(\mu,\nu)~with~ \mu+\nu \leq 1$ &  & $S_{Y}(A)=S_{Y}(B)=0.32 \Rightarrow A=B$ & $L(A)<L(B) \Rightarrow A<B$ & $C_1(A)<C_1(B) \Rightarrow A<B$ \\ \hline
Zhang and Yu \cite{Hz} \\
 $K(A_i,A^-)=\frac{\frac{1}{2}(c_i+d_i)}{\sqrt{\frac{{a_i}^2+{b_i}^2+{c_i}^2+{d_i}^2}{2}}}$ &  $A=([a_1,a_1],[a_1, a_1])$ , & $A=([0.2,0.2],[0.2,0.2])$  & $L(A)=0.04$ &$C_1(A)=0.04,$\\ $K(A_i,A^+)=\frac{\frac{1}{2}(a_i+b_i)}{\sqrt{\frac{{a_i}^2+{b_i}^2+{c_i}^2+{d_i}^2}{2}}}$ & $B=([b_1,b_1],[b_1, b_1]) ~with~ a_1>b_1$ & $ B=([0.3,0.3],[0.3,0.3])$ & $L(B)=0.09 $ & $C_1(B)=0.09$\\   Where $A^-=([0,0],[1,1]),$ & $K(A,A^-)=K(B,A^-)=\frac{1}{\sqrt{2}},$ &$K(A,A^-)=K(B,A^-)=\frac{1}{\sqrt{2}},$ & $L(A)<L(B) \Rightarrow A<B$ &$ C_1(A)<C_1(B) \Rightarrow A<B$ 
\\ $A^+=([1,1],[0,0])$ & $ K(A,A^+)=K(B,A^+)=\frac{1}{\sqrt{2}} \Rightarrow A=B$ & $ K(A,A^+)=K(B,A^+)=\frac{1}{\sqrt{2}} \Rightarrow A=B$ & &  \\ \hline
	 
\end{tabular}} 
\label{t4}
\end{table*} 

\subsection{Trapezoidal Intuitionistic Fuzzy Information System(TrIFIS)}
$~~~~~$Information system (IS) is a decision model used to select the best alternative from the all the alternatives in hand under various attributes. The data collected from the experts may be incomplete or imprecise numerical quantities. To deal with such data the thoery of IFS provided by Atanassov \cite{1} aids better. In information system, dominance relation rely on ranking of data, ranking of intuitionistic fuzzy numbers is inevitable.

\begin{defn}
An Information System \\$S=(U,AT,V,f)$ with $V=\cup _ {a \in AT } V_a$ where $V_a$ is a domain of attribute $a$ is called trapezoidal intuitionistic fuzzy information system  if $V$ is a set of TrIFN. We denote $f(x,a) \in V_a$ by $f(x,a)=\left\langle (a_1,a_2,a_3,a_4),({a_1}{'},{a_2}{'},{a_3}{'},{a_4}{'})\right\rangle$ where $a_i,{a_i}^{'} \in [0,1] $.
\end{defn}
The numerical illustration is given in example \ref{eg2}.
\begin{defn}
An TrIFIS, $S=(U,AT,V,f)$ together with weights $W = \{w_a / a \in AT \}$ is called  Weighted  Trapezoidal Intuitionistic Fuzzy Information System (WTrIFIS) and is denoted by \\$S=(U,AT,V,f, W)$.
\end{defn}
\begin{defn}
Let $a \in AT$ be a criterion. Let $x,y \in U$.  If $f(x,a)>f(y,a)$ (as per definition \ref{rank}) then $x>_a y$ which indicates that $x$ is better than (outranks) $y$ with respect to the criterion $a$.  Also $x=_a y$ means that $x$ is equally good as $y$ with respect to the criterion $a$, if $f(x,a)=f(y,a)$.
\end{defn}

\begin{defn}
 Let $S=(U,AT,V,f, W)$ be an WTrIFIS and $A\subseteq AT$. \\Let ${B_A}(x,y)=\left\{a \in A | x>_a y\right\}$ and let ${C_A}(x,y)=\left\{ a \in A | x=_ay \right\}$. The weighted fuzzy dominance relati-\\on ${WR_A}(x,y):U \times U \rightarrow [0,1]$ is defined by ${WR_A}(x,y)=\sum_{a \in {B_A}(x,y)} w_a + \frac{\sum_{a \in{C_A}(x,y)} w_a}{2}$.
\label{df2}
\end{defn}
\begin{defn} Let $S=(U,AT,V,f, W)$ be an WTrIFIS and $A\subseteq AT$. The entire dominance degree of each object is defined as \\${WR_A}(x_i)=\frac{1}{\left|U\right|} \sum ^{\left|U\right|}_{j=1}{WR_A}(x_i,y_j)$
\end{defn}
\subsection{Algorithm for Ranking of objects in WTrIFIS}
Let $S=(U,AT,V,f, W)$ be an WTrIFIS. The objects in $U$ are ranked using following algorithm.\\ \\
\textbf{Algorithm:5.3 \label{alg}}\\
1. Using  Definition \ref{rank} find $C_j's$ accordingly, to decide whether  ${x_i}~ {>_a}~{x_i} ~or~{x_j}~{>_a} ~{x_i} ~or~ {x_i}~ {=_a}~ {x_j}$ for all $a \in A (A \subseteq AT)$ and for all $x_i,x_j \in U$.\\
2. Enumerate $B_A(x_i,x_j)$ using \\${B_A}(x_i,x_j)=\left\{a \in A | x_i>_a x_j\right\}$ and $C_A(x_i,x_j)$ using ${C_A}(x_i,x_j)= \left\{ a \in A | x_i  =_a x_j \right\}$.\\
3.Calculate the weighted fuzzy dominance relation using ${WR_A}(x,y):U \times U \rightarrow [0,1]$ defined by ${WR_A}(x_i,x_j)=\sum_{a \in {B_A}(x_i,x_j)} w_a + \frac{\sum_{a \in{C_A}(x_i,x_j)} w_a}{2}$.\\
4. Calculate the entire dominance degree of each object using ${WR_A}(x_i)=\frac{1}{\left|U\right|} \sum ^{\left|U\right|}_{j=1}{WR_A}(x_i,x_j)$.\\
5. The objects are ranked using entire dominance degree. The larger the value of ${WR_A}(x_i)$, the better is the object.
\subsection{Numerical Illustration}
In this subsection, Algorithm \ref{alg} is illustrated by an example \ref{eg2}. 
\begin{ex}In this example we consider a selection problem of the best supplier for an automobile company from the available alternatives $\left\{x_i | i=1~to~10\right\}$ of pre evaluated 10 suppliers, based on WTrIFIS with attributes $\left\{a_j|j= 1 ~to~5\right\}$ as product quality, relation-\\ship closeness, delivery performance, social responsib-\\ility and legal issue.\\ 
An TrIFIS with $U=\left\{x_1,x_2,...,x_{10}\right\}$, $AT=\{a_1,a_2,...\\,a_5\}$ is given in Table \ref{t1}, and weights for the each attribute $W_{a}$ is given by $W=\left\{w_{a} | a \in AT\right\}=\left\{0.3, 0.2, 0.15, 0.17, 0.18 \right\} $.

\label{eg2}
\end{ex}

\begin{table*}[ht]
\caption{WTrIFIS to evaluate alternatives with respect to criteria} 
\centering 
\scalebox{0.9}{

\tiny

\begin{tabular}{c c c c c c } 
\hline 
$ $ & $a_1$ & $a_2$ & $a_3$ & $a_4$ & $a_5$ \\\hline  
$x_1$ & $\left\langle 0.2,0.4\right\rangle$ & $\left\langle [0,0.2],[0.2,0.3]\right\rangle$ & $\langle(0.1,0.2,0.3,0.4)$, & $\langle (0.2,0.2,0.3,0.3),$ & $\left\langle [0.2,0.4],[0,0.2]\right\rangle$ \\
 &  & & $(0.3, 0.4,0.5,0.6)\rangle$ & $(0.4,0.6,0.7,0.8)\rangle$ & \\ \\
 
$x_2$ & $\langle (0.1,0.1,0.15,0.2),$ & $\langle (0.1,0.2,0.3,0.3),$ &$\left\langle [0.2,0.4],0.6 \right\rangle$ & $\left\langle 0.4,0.4 \right\rangle$ &  $\left\langle [0.4,0.6],[0,0.2]\right\rangle$ \\  \\
 & $(0.16,0.23,0.37,0.5)\rangle$ &  $( 0.35,0.45,0.50,0.70)\rangle$ &  &  & \\ \\ \\

$x_3$ & $[0,0.2],[0.2,0.3]$ & $\left\langle[0.2,0.4],0.2 \right\rangle$ & $\langle(0.1,0.3,0.4,0.5)$,  & $\left\langle 0.4,[0.2,0.4]\right\rangle$ & $\left\langle 0.2,[0.4,0.8]\right\rangle$\\ 
 & & & $(0.45, 0.60,0.60,0.75)\rangle$ &  & \\  \\
$x_4$ & $\left\langle 0, [0,0.6] \right\rangle$ &  $\langle (0,0.2,0.3,0.35),$ & $\langle (0.1,0.2,0.3,0.4)$, &$\langle (0.1,0.2,0.3,0.4)$, &$\left\langle 0.2,0.6\right\rangle$   \\ 
& &$( 0.30,0.45,0.50,0.60)\rangle$ &$ (0.35,0.45,0.55,0.65)\rangle$ &$(0.3, 0.4,0.5,0.6)\rangle $  & \\ \\

$x_5$ & $\langle (0,0.1,0.15,0.25 ),$ &  $\left\langle [0.4,0.6],[0.2,0.4] \right\rangle$ & $\langle (0.25,0.30,0.35,0.45 ),$ & $\left\langle [0.2,0.4],[0.4,0.6] \right\rangle$ & $\left\langle [0.2,0.4],0.2\right\rangle$  \\ 
&$(0.20,0.26,0.37,0.60 )\rangle $ &  &$( 0.35,0.45,0.50,0.60 )\rangle$ &   &  \\ \\

$x_6$ & $\left\langle [0.2,0.4],[0,0.2]\right\rangle$  &$\langle (0.10,0.10,0.15,0.20),$ & $\left\langle [0.2,0.6],[0,0.2]\right\rangle$ & $\langle (0,0.20,0.30,0.42),$ & $\left\langle 0.7,[0,0.4]\right\rangle$  \\ & &$(0.16,0.23,0.37,0.50)\rangle$ & & $(0.50,0.60,0.70,0.90)\rangle$ & \\  \\

$x_7$ & $\left\langle 0.2,0.7 \right\rangle$ & $\left\langle 0.2,0.6 \right\rangle$ & $\langle (0.20,0.30,0.40,0.60)$, & $\left\langle 0.2,0.7 \right\rangle$ & $\left\langle 0.4,0.2 \right\rangle$ \\
 & & & $(0.50,0.60,0.60,0.90)\rangle$ &  & \\  \\
  
$x_8$ & $\left\langle [0.2,0.4],[0,0.1] \right\rangle$ & $\left\langle [0.4,0.6],[0,0.2]\right\rangle$ & $\langle (0,0.30,0.35,0.40 ),$ & $ \left\langle 0.3,0.2 \right\rangle$ & $\langle (0,0.2,0.3,0.3)$, \\ 
& & &$( 0.40,0.45,0.50,0.55) \rangle$ &  & $(0.35,0.37,0.42,0.50)\rangle$ \\ \\

$x_9$ &$\left\langle [0.2,0.4],[0.4,0.6]\right\rangle$ & $\left\langle 0.4,0.4 \right\rangle$ & $\langle (0.05,0.21,0.34,0.35),$ & $\left\langle [0.3,0.4],0.2\right\rangle$ &  $\left\langle 0.6, 0.2 \right\rangle$ \\ & & &$( 0.37,0.52,0.63,0.79 )\rangle$ &  & \\ \\

$x_{10}$ & $\left\langle 0.7,[0.2,0.6]\right\rangle$ & $\left\langle 0.3,0.6 \right\rangle$ & $\langle (0.13,0.21,0.34,0.45 ),$ & $\left\langle 0.4,0.2\right\rangle$ & $\langle (0.1,0.2,0.3,0.35)$, \\ 
& & & $( 0.45,0.52,0.63,0.85) \rangle$&  & $(0.37,0.37,0.42,0.45)\rangle$ \\ \\
 \hline
\end{tabular} }

\label{t1}
\end{table*} 
In the table 2, $f(x_i, a_i) = (a_1, a_2, a_3, a_4),( c_1,c_2,c_3\\,c_4)$ denotes the trapezoidal intuitionistic fuzzy numbers which include intuitionistic fuzzy values and interval valued intuitionstic fuzzy numbers and they stand for the evaluation of alternative $x_i$ under the criteria $a_i$ with acceptance of $(a_1, a_2, a_3, a_4)$ and nonacceptance of $( c_1,c_2,c_3,c_4)$. For example, $f(x_1, a_1)$ denotes supplier $x_1$ is evaluated  under the criteria 'product quality' $(a_1)$ with "$20 $\%  of  acceptance and $40$\% of non acceptance" , $f(x_5, a_3)$ denotes the supplier $x_5$  is evaluated  under the criteria 'delivery performance $(a_3)$' with "around $30$\% to $35$\% of  accepatance and around $45$\%  to $50$\% of non acceptance" and  $f(x_9, a_1)$ denotes the supplier $x_9$  is evaluated  under the criteria 'product quality' $(a_1)$ with " $20$\% to 40\% of  acceptance and $40$\%  to $60$\% of non acceptance".\\
Step:1
For $i=1, (\alpha,\beta)=(1,1)$:  By step 1,  $C_1(f(x_i,a_j))$ using defintion \ref{rank} and note \ref{nt 1}, for all $a_i \in AT$ and for all $x_i \in U$ is found and tabulated in table \ref{t2}. If $C_1(f(x_i,a_j))=C_1(f(x_j,a_j))$ for any alternatives $x_i,x_j$ then $C_2$ and other necessary score functions ($C_3 ~and~ C_4$) are found wherever required. The bold letters are used in table \ref{t2} to represent the equality of scores. From table \ref{t2} we observe that in many places $C_j's$ are not distiguishable for different IFNs.  

\begin{table}[htb]
\caption{$C_{1},C_{2},C_{3},~and~C_{4}$  for $i=1, (\alpha,\beta)=(1,1)$} 

\centering 
\small  
 
\scalebox{0.422}{
\begin{tabular}{l l l l l l l l l l l l} 
\hline 
$ $ & $a_1$ & $a_2$ & $a_3$ & $a_4$ & $a_5$ & $ $ & $a_1$ & $a_2$ & $a_3$ & $a_4$ & $a_5$ \\ \\ \hline
$C_1$ & & & &  & & $C_2$ \\ \\
$x_1$ & $\textbf{-0.12}$ & $\textbf{-0.12}$ & $-0.085$ & $\textbf{-0.235}$ & $+0.24$ & $$ & $-0.28$ & $-0.18$ & $$ & $\textbf{-0.565}$ & $$ \\ \\
$x_2$ & $\textbf{-0.14925}$ & $\textbf{-0.105}$ & $-0.12$ & $\textbf{+0.16}$ & $+0.46$ & $$ & $\textbf{-0.23075}$ & $\textbf{-0.345}$ & $$ & $-0.16$ & $$ \\ \\
$x_3$ & $\textbf{-0.12}$ & $\textbf{+0.16}$ & $\textbf{-0.04}$ & $\textbf{+0.22}$ & $\textbf{-0.28}$ & $$ & $-0.18$ & $+0.04$ & $\textbf{-0.46}$ & $-0.02$ & $\textbf{-0.52}$\\ \\
$x_4$ & $-0.3$ & $\textbf{-0.105}$ & $-0.1225$ & $-0.085$ & $\textbf{-0.28}$ & $$ & $$ & $\textbf{-0.345}$ & $$ & $$ & $\textbf{-0.52}$\\ \\
$x_5$ & $\textbf{-0.14925}$ & $+0.36$ & $\textbf{+0.005}$ & $-0.04$ & $+0.16$ & $$ & $\textbf{-0.23075}$ & $$ & $\textbf{-0.305}$ & $$ & $$\\ \\
$x_6$ & $+0.24$ & $-0.13575$ & $+0.36$ & $\textbf{-0.235}$ & $+0.64$ & $$ & $$ & $$ & $$ & $\textbf{-0.565}$ & $$\\ \\
$x_7$ & $-0.36$ & $-0.28$ & $\textbf{-0.04}$ & $-0.36$ & $+0.28$ & $$ & $$ & $$ & $\textbf{-0.46}$ & $$ & $$\\ \\
$x_8$ & $+0.27$ & $+0.46$ & $\textbf{+0.005}$ & $\textbf{+0.16}$ & $\textbf{-0.045}$ & $$ & $$ & $$ & $\textbf{-0.305}$ & $+0.04$ & $\textbf{-0.245}$\\ \\
$x_9$ & $-0.04$ & $\textbf{+0.16}$ & $\textbf{-0.1383}$ & $\textbf{+0.22}$ & $+0.52$ & $$ & $$ & $-0.16$ & $\textbf{-0.4617}$ & $+0.08$ & $$\\ \\
$x_{10}$ & $0.58$ & $\textbf{-0.12}$ & $\textbf{-0.1383}$ & $+0.28$ & $\textbf{-0.045}$ & $$ & $$ & $-0.48$ & $\textbf{-0.4617}$ & $$ & $\textbf{-0.245}$\\ \\
$C_3$ & $$ & $$ & $$ & $$ & $$ & $C_4$ & $$ & $$ & $$ & $$ & $$   \\ \\
$x_1$ & $$ & $$ & $$ & $\textbf{+0.165}$ & $$  & $$ & $$ & $$ & $$ & $\textbf{-0.145}$ & $$\\ \\
$x_2$ & $\textbf{+0.07075}$ & $\textbf{-0.095}$ & $$ & $$ & $$ & $$ & $\textbf{-0.09475}$ & $\textbf{-0.105}$ & $$ & $$ & $$\\ \\
$x_3$ & $$ & $$ & $\textbf{+0.16}$ & $$ & $+0.32$ & $$ & $$ & $$ & $\textbf{-0.08}$ & $$ & $$\\ \\
$x_4$ & $$ & $\textbf{-0.095}$ & $$ & $$ & $+0.12$ & $$ & $$ & $\textbf{-0.105}$ & $$ & $$ & $$\\ \\
$x_5$ & $\textbf{+0.07075}$ & $$ & $\textbf{+0.155}$ & $$ & $$ & $$ & $\textbf{-0.09475}$ & $$ & $\textbf{-0.07}$ & $$ & $$\\ \\
$x_6$ & $$ & $$ & $$ & $\textbf{+0.165}$ & $$ & $$ & $$ & $$ & $$ & $\textbf{-0.0145}$ & $$\\ \\
$x_7$ & $$ & $$ & $\textbf{+0.16}$ & $$ & $$ & $$ & $$ & $$ & $\textbf{-0.08}$ & $$ & $$\\ \\
$x_8$ & $$ & $$ & $\textbf{+0.155}$ & $$ & $\textbf{+0.075}$ & $$ & $$ & $$ & $\textbf{-0.07}$ & $$ & $\textbf{-0.101}$\\ \\
$x_9$ & $$ & $$ & $\textbf{+0.1517}$ & $$ & $$ & $$ & $$ & $$ & $\textbf{-0.1725}$ & $$ & $$\\ \\
$x_{10}$ & $$ & $$ & $\textbf{+0.1517}$ & $$ & $\textbf{+0.075}$ & $$ & $$ & $$ & $\textbf{-0.1725}$ & $$ & $\textbf{-0.101}$ \\ \\

\hline
\end{tabular} }
\label{t2}
\end{table}

Hence from table \ref{t2} we do not get the best alternative. Therefore the same procedure which is explained in step 1 is repeated for $i=2, (\alpha,\beta)=(1/2,1/2)$ and it is shown in table \ref{t21}.

\begin{table}[htb]
\caption{$C_{5},C_{6},C_{7},~and~C_{8}$ for $i=2, (\alpha,\beta)=(1/2,1/2)$} 
\centering 
\small

\scalebox{0.48}{\begin{tabular}{l l l l l l l l l l l l} 
\hline 
$ $ & $a_1$ & $a_2$ & $a_3$ & $a_4$ & $a_5$ & $ $ & $a_1$ & $a_2$ & $a_3$ & $a_4$ & $a_5$ \\ \\ \hline
$C_5$ & & & &  & & $C_6$ \\ \\
$x_1$ & $\textbf{-0.12}$ & $\textbf{-0.12}$ & $-0.0775$ & $-0.2125$ & $+0.24$ & $$ & $-0.28$ & $-0.18$ & $$ & $$ & $$ \\ \\
$x_2$ & $-0.13644$ & $-0.155$ & $-0.12$ & $\textbf{+0.16}$ & $+0.46$ & $$ & $$ & $$ & $$ & $-0.16$ & $$ \\ \\
$x_3$ & $\textbf{-0.12}$ & $\textbf{+0.16}$ & $-0.07063$ & $\textbf{+0.22}$ & $\textbf{-0.28}$ & $$ & $-0.18$ & $+0.04$ & $$ & $-0.02$ & $\textbf{-0.52}$\\ \\
$x_4$ & $-0.3$ & $-0.14188$ & $-0.115$ & $-0.0775$ & $\textbf{-0.28}$ & $$ & $$ & $$ & $$ & $$ & $\textbf{-0.52}$\\ \\
$x_5$ & $-0.17825$ & $+0.36$ & $+0.0275$ & $-0.04$ & $+0.16$ & $$ & $$ & $$ & $$ & $$ & $$\\ \\
$x_6$ & $+0.24$ & $-0.12969$ & $+0.36$ & $-0.2735$ & $+0.64$ & $$ & $$ & $$ & $$ & $$ & $$\\ \\
$x_7$ & $-0.36$ & $-0.28$ & $-0.01875$ & $-0.36$ & $+0.28$ & $$ & $$ & $$ & $$ & $$ & $$\\ \\
$x_8$ & $+0.27$ & $+0.46$ & $-0.08219$ & $\textbf{+0.16}$ & $-0.123$ & $$ & $$ & $$ & $$ & $+0.04$ & $$\\ \\
$x_9$ & $-0.04$ & $\textbf{+0.16}$ & $-0.1886$ & $\textbf{+0.22}$ & $+0.52$ & $$ & $$ & $-0.16$ & $$ & $+0.08$ & $$\\ \\
$x_{10}$ & $+0.58$ & $\textbf{-0.12}$ & $-0.14263$ & $+0.28$ & $-0.06656$ & $$ & $$ & $-0.48$ & $$ & $$ & $$\\ \\
$C_7$ & $$ & $$ & $$ & $$ & $$ & $C_8$ & $$ & $$ & $$ & $$ & $$   \\ \\
$x_1$ & $$ & $$ & $$ & $$ & $$  & $$ & $$ & $$ & $$ & $$ & $$\\ \\
$x_2$ & $$ & $$ & $$ & $$ & $$ & $$ & $$ & $$ & $$ & $$ & $$\\ \\
$x_3$ & $$ & $$ & $$ & $$ & $+0.32$ & $$ & $$ & $$ & $$ & $$ & $$\\ \\
$x_4$ & $$ & $$ & $$ & $$ & $+0.12$ & $$ & $$ & $$ & $$ & $$ & $$\\ \\
$x_5$ & $$ & $$ & $$ & $$ & $$ & $$ & $$ & $$ & $$ & $$ & $$\\ \\
$x_6$ & $$ & $$ & $$ & $$ & $$ & $$ & $$ & $$ & $$ & $$ & $$\\ \\
$x_7$ & $$ & $$ & $$ & $$ & $$ & $$ & $$ & $$ & $$ & $$ & $$\\ \\
$x_8$ & $$ & $$ & $$ & $$ & $$ & $$ & $$ & $$ & $$ & $$ & $$\\ \\
$x_9$ & $$ & $$ & $$ & $$ & $$ & $$ & $$ & $$ & $$ & $$ & $$\\ \\
$x_{10}$ & $$ & $$ & $$ & $$ & $$ & $$ & $$ & $$ & $$ & $$ & $$ \\ \\
\hline
\end{tabular} }
\label{t21}
\end{table}
The weighted fuzzy dominance relation using \\${WR_A}(x,y) = \sum_{a \in {B_A}(x,y)} w_a + \frac{\sum_{a \in{C_A}(x,y)} w_a}{2}$ is calculated and is tabulated in table \ref{t3}. For example, ${B_A}(x_1,x_2)=\left\{a_1,a_2,a_3\right\}$ and ${C_A}(x_1,x_2)=\left\{\right\}$ and hence ${WR_A}(x_1,x_2)=0.3+0.2+0.15=0.65$.
 \begin{table}[htb]
\caption{Weighted Fuzzy Dominance relation between two alternatives ${WR_A}(x,y)$} 
\centering 
\scalebox{0.66}{
\begin{tabular}{l l l l l l l l l l l } 
\hline
${WR_A}(x,y)$ & $x_1$ & $x_2$ & $x_3$ & $x_4$ & $x_5$ & $x_6$ & $x_7$ & $x_8$ & $x_9$ & $x_{10}$ \\ \\ \hline
$x_1$ & $0.50$ & $0.65$ & $0.18$ & $0.83$ & $0.48$ & $0.37$ & $0.67$ & $0.33$ & $0.15$ & $0.53$ \\ \\
$x_2$ & $0.35$ & $0.50$ & $0.18$ & $0.65$ & $0.65$ & $0.17$ & $0.85$ & $0.18$ & $0.15$ & $0.33$ \\ \\
$x_3$ & $0.82$ & $0.82$ & $0.50$ & $1$ & $0.47$ & $0.37$ & $0.67$ & $0.32$ & $0.35$ & $0.35$ \\ \\
$x_4$ & $0.17$ & $0.35$ & $0$ & $0.50$ & $0$ & $0.17$ & $0.67$ & $0$ & $0.15$ & $0.15$ \\ \\
$x_5$ & $0.52$ & $0.35$ & $0.53$ & $1$ & $0.50$ & $0.37$ & $0.82$ & $0.33$ & $0.35$ & $0.53$ \\ \\
$x_6$ & $0.63$ & $0.83$ & $0.63$ & $0.83$ & $0.63$ & $0.50$ & $1$ & $0.33$ & $0.63$ & $0.33$ \\ \\
$x_7$ & $0.33$ & $0.15$ & $0.33$ & $0.33$ & $0.18$ & $0$ & $0.5$ & $0.33$ & $0.15$ & $0.33$ \\ \\
$x_8$ & $0.67$ & $0.82$ & $0.68$ & $1$ & $0.67$ & $0.67$ & $0.67$ & $0.50$ & $0.65$ & $0.35$ \\ \\
$x_9$ & $0.85$ & $0.85$ & $0.65$ & $0.85$ & $0.65$ & $0.37$ & $0.85$ & $0.35$ & $0.50$ & $0.38$ \\ \\
$x_{10}$ & $0.47$ & $0.67$ & $0.65$ & $0.85$ & $0.47$ & $0.67$ & $0.67$ & $0.65$ & $0.62$ & $0.50$ \\ \\
\hline
\end{tabular} }
\label{t3}
\end{table}
Now the entire dominance degree of each object using ${WR_A}(x_i)=\frac{1}{\left|U\right|} \sum ^{\left|U\right|}_{j=1}{WR_A}(x_i,y_j)$ is found by definiton \ref{df2}.\\For example, ${WR_A}(x_1)=\frac{1}{10} \sum ^{10}_{j=1}{WR_A}(x_1,x_j)=0.469$.\\
 So by step 5, $x_8$ is selected as the best object from the weighted trapezoidal intuitionistic fuzzy information system is seen from table \ref{t4}.
\begin{table}[htb]
\caption{Total Dominance degree ${R_A}(x_i)$}  
\centering 
\scalebox{0.6}{
\begin{tabular}{l l l l l l l l l l l } 
\hline
$X_i$ & $x_1$ & $x_2$ & $x_3$ & $x_4$ & $x_5$ & $x_6$ & $x_7$ & $x_8$ & $x_9$ & $x_{10}$ \\ \\
${R_A}(x_i)$ & $0.469$ & $0.401$ & $0.567$ & $0.216$ & $0.530$ & $0.634$ & $0.263$ & $0.668$ & $0.630$ & $0.622$ \\
\hline
\end{tabular} }
\label{t4}
\end{table}   
\section{Conclusion}$~~~~~$In this paper, the total ordering on the set of all IFNs is achieved. The total orderings introduced and discussed in this paper are consistent with the natural ordering of real numbers. Actually this total ordering on IFNs generalises the total ordering on FNs defined by Wei Wang, Zhenyuan Wang \cite{34}. Therefore this is an appropriate generalization of the total ordering on the set of all real numbers to the set of IFNs. This method can order intuitionistic fuzzy numbers, either alone or as a supplementary means with other ranking methods, and may be adopted in decision making with fuzzy information.\\
\textbf{Acknowledgement}\\
The Authors thank the valuable suggestions from anonymous referees for the betterment of the quality of the paper.

\end{document}